\documentclass[dvipdfmx,a4,11pt]{amsart}
\usepackage{amsfonts,amssymb,amsmath,amsthm}
\usepackage{url}
\usepackage{color}
\usepackage{enumerate}
\usepackage{bm}
\usepackage{graphicx}
\usepackage{color}
\usepackage{enumerate}
\usepackage{multicol}
\usepackage{multirow}
\usepackage{array}
\usepackage[all]{xy}
\newtheorem{thm}{Theorem}[section]
\newtheorem{cor}[thm]{Corollary}
\newtheorem{lem}[thm]{Lemma}
\newtheorem{prop}[thm]{Proposition}

\theoremstyle{definition}
\newtheorem{defin}[thm]{Definition}
\newtheorem{rem}[thm]{Remark}

\numberwithin{equation}{section}

\author[M. Matsuno]{Masaki Matsuno}
\address{
	Graduate School of Integrated Science and Technology, 
	Shizuoka University\\
	Ohya 836, Shizuoka 422-8529, Japan}
\email{matsuno.masaki.14@shizuoka.ac.jp}


\newcommand{\Gl}{{\rm GL}}

\newcommand{\Aut}{{\rm Aut}}

\newcommand{\al}{\alpha}
\newcommand{\be}{\beta}
\newcommand{\ga}{\gamma}

\newcommand{\la}{\lambda}
\newcommand{\ep}{\varepsilon}

\date{\today}
\begin{document}
	\title[$3$-dimensional quadratic AS-regular algebras of Type EC]
	{A complete classification of $3$-dimensional quadratic AS-regular algebras of Type EC}

\begin{abstract}
	Classification of AS-regular algebras is one of the main interests in noncommutative
	algebraic geometry.
	We say that a $3$-dimensional quadratic AS-regular algebra is of Type EC if its point scheme
	is an elliptic curve in $\mathbb{P}^{2}$.
	In this paper, we give a complete list of geometric pairs and a complete list of
	twisted superpotentials corresponding to such algebras.
	As an application, we show that there are only two exceptions
	up to isomorphism among all $3$-dimensional quadratic AS-regular algebras
	 which cannot be written as a twist of a Calabi-Yau AS-regular algebra
	 by a graded algebra automorphism.
\end{abstract}
\maketitle
\section{Introductuion}
In noncommutative algebraic geometry,
an Artin-Schelter regular (AS-regular)
algebra defined by Artin-Schelter \cite{AS} is one of the most important classes
of algebras to study.
Classification of AS-regular algebras is one of the important problems in this field.
In fact, Artin-Tate-Van den Bergh \cite{ATV1} proved that
there exists a one-to-one correspondence
between $3$-dimensional AS-regular algebras and regular geometric pairs
by using algebraic geometry, which
was a starting point of noncommutative algebraic geometry.
In this paper, we say that a $3$-dimensional quadratic AS-regular algebra is of Type EC
(or simply a Type EC algebra) if its point scheme is an elliptic curve in $\mathbb{P}^{2}$.
A typical example is a $3$-dimensional Sklyanin algebra (Type A algebra),
which had been intensively studied in many literature, but
there are others (Type B, E, H algebras), which have been much less paid attention to.

A list of geometric pairs for \lq\lq generic\rq\rq $3$-dimensional AS-regular algebras was given in
\cite[4.13]{ATV1}.
It is known that there are many \lq\lq non-generic\rq\rq $3$-dimensional
AS-regular algebras (see, for example, \cite[Theorem 3.1]{IM1}).
Restricting to Type EC algebras, we have the following questions:
\begin{enumerate}[{\rm (1)}]
	\item Is there a \lq\lq non-generic\rq\rq $3$-dimensional AS-regular algebra of Type EC?
	\item Is every geometric pair given in the list of \cite[4.13]{ATV1} in fact regular?
	\item How many algebras in each type up to isomorphism?
\end{enumerate}
To resolve these questions, we redefine Type B, E, H algebras in this paper and
give a more complete list of geometric pairs for Type EC algebras.
In particular, we show that there are only two Type E algebras and only two Type H algebras
up to isomorphism.

 
Recently, Dubois-Violette \cite{D} and Bocklandt-Schedler-Wemyss \cite{BSW}
showed that every $3$-dimensional quadratic AS-regular algebra $A$ 
is isomorphic to a derivation-quotient algebra $\mathcal{D}(w)$ 
of a twisted superpotential $w$,
and Mori-Smith \cite{MS2} classified $3$-dimensional quadratic 
Calabi-Yau AS-regular algebras by using superpotentials.
In this paper, we give a list of twisted superpotentials for Type EC algebras.
Combining with the list of twisted superpotentials for non-Type EC algebras
given in \cite[Table 1]{IM2}, we finally complete the project started by \cite{AS},
the classification of $3$-dimensional AS-regular algebras up to isomorphism, in the quadratic case.

It was proved in \cite[Corollary 3.7]{IM2} that every non-Type EC algebra can be
written as a twist of a Calabi-Yau AS-regular algebra by a graded algebra automorphism.
As an application of our classification, we show that there are only two
exceptions up to isomorphism 
among all $3$-dimensional quadratic AS-regular algebras
for this property to hold:
\begin{thm}[{\rm Theorem \ref{main05}}]\label{thm}
	A $3$-dimensional quadratic AS-regular algebra is a twist of a Calabi-Yau AS-regular
	algebra by a graded algebra automorphism if and only if it is not of Type E.
\end{thm}

This paper is organized as follows:
In section $2$, we recall the definitions of an AS-regular algebra from \cite{AS},
a Calabi-Yau algebra from \cite{G}, a twisted superpotential and a twist of 
a superpotential in the sense of \cite{MS1}.
Also, we recall the definitions of  a geometric algebra 
for quadratic algebras from \cite{Mo}.
We also recall some properties of
elliptic curves in $\mathbb{P}^{2}$ and
the results of the previous paper \cite{IM1}.
In section $3$, we give a complete list of geometric pairs for
Type EC algebras up to isomorphism (Theorem \ref{main03}, Table 1).
Finally, in section $4$, we give a complete list of twisted superpotentials
for Type EC algebras up to isomorphism (Theorem \ref{main04}, Table 2)
and we prove Theorem \ref{thm}.
\section{Preliminaries}
Throughout this paper, 
let $k$ be an algebraically closed field of 
characteristic $0$. 
A graded $k$-algebra means an $\mathbb{N}$-graded algebra
$A=\bigoplus_{i\in \mathbb{N}}A_{i}$. 
A {\it connected graded} $k$-algebra $A$ is 
a graded $k$-algebra such that $A_{0}=k$. 
\subsection{AS-regular algebras and Calabi-Yau algebras }\label{sec2}
Let $A$ be a connected graded $k$-algebra. 
The {\it Gelfand-Kirillov dimension} of $A$ is defined by
${\rm GKdim}\,A:=
{\rm inf}\{\alpha\in \mathbb{R} \mid {\rm dim}_{k}(\sum_{i=0}^{n}A_{i})\leq n^{\alpha}
\text{ for all }n \gg 0\}$.
\begin{defin}[\cite{AS}]
	A connected graded algebra $A$ is called a {\it $d$-dimensional Artin-Schelter regular 
		{\rm (}AS-regular{\rm )} algebra}  
	if $A$ satisfies the following conditions: 
	\begin{enumerate}[(i)]
		\item ${\rm gldim}\,A=d<\infty$, 
		\item ${\rm GKdim}\,A<\infty$, 
		\item ({\it Gorenstein condition})\quad
		${\rm Ext}_{A}^{i}(k,A)\cong \left\{
		\begin{array}{ll}
		k   &\quad (i=d), \\
		0   &\quad (i\neq d). 
		\end{array}
		\right.$
	\end{enumerate}
\end{defin}
Any $3$-dimensional AS-regular algebra $A$ 
generated in degree 1 is
either a quadratic or a cubic algebra
(\cite[Theorem 1.5 (i)]{AS}). 
In this paper, we focus on $3$-dimensional quadratic AS-regular algebras. 

Here, we recall the definition of a Calabi-Yau algebra introduced by \cite{G}. 
\begin{defin}[\cite{G}]
	A $k$-algebra $S$ is called {\it $d$-dimensional Calabi-Yau} 
	if $S$ satisfies the following conditions: 
	\begin{enumerate}[(i)]
		\item $\mathrm{pd}_{S^{\rm e}}S=d<\infty$, 
		\item $
		\mathrm{Ext}_{S^{\rm e}}^{i}(S,S^{\rm {e}})\cong
		\begin{cases}
		S & \text{ if } i=d, \\
		0 & \text{ if } i \neq d
		\end{cases}
		$ \quad (as $S^{\rm e}$-modules)
	\end{enumerate}
	where $S^{\rm e}=S \otimes_{k} S^{{\rm op}}$ 
	is the enveloping algebra of $S$. 
\end{defin}
\begin{rem}[{\cite[Lemma 1.2]{RRZ}}]
	Let $A$ be a connected graded $k$-algebra of finite GK-dimension.
	If $A$ is a $d$-dimensional Calabi-Yau algebra, then
	it is a $d$-dimensional AS-regular algebra.
\end{rem}
\subsection{Geometric algebras}
Let $V$ be a finite dimensional $k$-vector space.
For a quadratic algebra $A=T(V)/(R)$
where $R \subset V \otimes_{k} V$,
we set 
$$
\mathcal{V}(R):=\{(p,q)\in \mathbb{P}(V^{\ast})\times \mathbb{P}(V^{\ast})
\mid f(p,q)=0 \text{ for all } f\in R\}. 
$$

In \cite{Mo}, Mori introduced a {\it geometric algebra} over $k$ as follows. 
\begin{defin}[\cite{Mo}]
	\label{geometric}
	A geometric pair $(E,\sigma)$ consists of a projective variety 
	$E \subset \mathbb{P}(V^{\ast})$
	and $\sigma \in {\rm Aut}_{k}\,E$.
	Let $A=T(V)/(R)$ be a quadratic algebra with $R \subset V \otimes_{k} V$.
	\begin{enumerate}[{\rm (1)}]
		\item We say that $A$ satisfies (G1) if there exists a geometric pair $(E,\sigma)$ such that
		$$
		\mathcal{V}(R)=\{ (p,\sigma(p)) \in \mathbb{P}(V^{\ast})\times\mathbb{P}(V^{\ast})\,|\,p \in E \}.
		$$
		In this case, we write $\mathcal{P}(A)=(E,\sigma)$.
		\item We say that $A$ satisfies (G2) if there exists a geometric pair $(E,\sigma)$ such that
		$$
		R=\{ f \in V \otimes_{k} V\,|\,f(p,\sigma(p))=0\,\,{\rm for\,\,any\,\,} p \in E \}.
		$$
		In this case, we write $A=\mathcal{A}(E,\sigma)$.
		\item A quadratic algebra $A$ is called geometric if $A$ satisfies both (G1) and (G2)
		with $A=\mathcal{A}(\mathcal{P}(A))$.
	\end{enumerate}
	When $A=\mathcal{A}(E,\sigma)$ is a geometric algebra, $E$ is called the {\it point scheme} of $A$.
\end{defin}
\begin{thm}[\cite{ATV1}]
	\label{thmATV1}
	Every $3$-dimensional quadratic AS-regular algebra $A$ 
	is geometric. 
	Moreover, the point scheme $E$ of $A$ is either $\mathbb{P}^{2}$ or 
	a cubic divisor in $\mathbb{P}^{2}$. 
\end{thm}
\begin{rem}
	In the above theorem, $E\subset \mathbb{P}^{2}$ could be a 
	non-reduced cubic divisor in $\mathbb{P}^{2}$.
	See \cite[Definition 4.3]{Mo} for the definition of a geometric algebra 
	in the case that $E$ is non-reduced. 
\end{rem}

\subsection{Derivation-quotient algebras}
Now, we recall the definitions of a superpotential, 
a twisted superpotential and a derivation-quotient algebra from 
\cite{BSW} and \cite{MS1}. 
Also, we recall the definition of a {\it twist of a superpotential} 
due to \cite{MS1}. 

Let $V$ be a $3$-dimensional $k$-vector space and
we fix a basis $\{ x_{1},x_{2},x_{3} \}$ for $V$.
In this paper, 
we call an element $w \in V^{\otimes 3}$ a {\it potential}.
For $w\in V^{\otimes 3}$, 
there exist unique $w_{i}\in V^{\otimes 2}$ such that 
$w=\sum_{i=1}^{3}x_{i}\otimes w_{i}$. 
Then the {\it left partial derivative} of $w$ with respect to $x_{i}$ 
($i=1,2,3$) is
$$
\partial_{x_{i}}(w):=w_{i},
$$ 
and the {\it derivation-quotient algebra} of $w$ is
$$
\mathcal{D}(w):=T(V)/(\partial_{x_{1}}(w),\partial_{x_{2}}(w),\partial_{x_{3}}(w)).
$$
Since, for $w\in V^{\otimes 3}$, 
there exist unique $w'_{i}\in V^{\otimes 2}$ such that 
$w=\sum_{i=1}^{3}w'_{i}\otimes x_{i}$,
we can also define the {\it right partial derivative} of $w$ with respect to $x_{i}$
($i=1,2,3$) by
$$
(w)\partial_{x_{i}}:=w'_{i}.
$$ 


We define the $k$-linear map $\varphi$: 
$V^{\otimes 3}\longrightarrow V^{\otimes 3}$ by 
$$
\varphi(v_{1}\otimes v_{2}\otimes v_{3}):=v_{3}\otimes v_{1}\otimes v_{2}. 
$$
We write ${\rm GL}(V)$ for the general linear group of $V$. 
\begin{defin}[\cite{BSW}, \cite{MS1}]
	Let $w \in V^{\otimes 3}$ be a potential.
	\begin{enumerate}
		\item If $\varphi(w)=w$ for $w\in V^{\otimes 3}$, 
		then $w$ is called a {\it superpotential}. 
		\item
		If there exists $\theta\in {\rm GL}(V)$ such that 
		$$
		(\theta\otimes {\rm id}_{V}\otimes {\rm id}_{V})\varphi(w)=w,
		$$
		then $w$ is called a {\it twisted superpotential}.
	\end{enumerate}
\end{defin}
\begin{lem}[{\cite[Theorem 4.13]{MS1}}]\label{TSP}
	Let $w \in V^{\otimes 3}$ be a potential.
	Assume that the left derivatives $\partial_{x_{1}}(w), \partial_{x_{2}}(w), \partial_{x_{3}}(w)$
	are linearly independent.
	Then $w$ is a twisted superpotential if and only if
	there is $Q \in \Gl_{3}(k)$ such that 
	$$((w)\partial_{x_{1}}, (w)\partial_{x_{2}}, (w)\partial_{x_{3}})
	=Q(\partial_{x_{1}}(w),\partial_{x_{2}}(w), \partial_{x_{3}}(w)).$$ 
\end{lem}
We recall the notion of twisted algebra by a graded algebra
automorphism.
We denote by ${\rm GrAut}_{k}\,A$ the group of
graded algebra automorphisms of $A$.

Let $\phi \in {\rm GrAut}_{k}\,A$.
A new graded and associative multiplication 
$\ast$ on the underlying graded $k$-vector space
$\bigoplus_{i\in \mathbb{N}}A_{i}$ 
is defined by 
$$
a\ast b:=a \phi^{l}(b)\quad\text{for all }a \in A_{l}, b\in A_{m}. 
$$
The graded $k$-algebra $(\bigoplus_{i\in \mathbb{N}}A_{i}, \ast)$ 
is called the {\it twisted algebra} of $A$ 
by $\phi$ and is denoted by $A^{\phi}$.
\begin{defin}[\cite{MS1}]
	For a superpotential $w\in V^{\otimes 3}$ and 
	$\theta \in {\rm GL}(V)$,  
	$$
	w^{\theta}:=(\theta^{2}\otimes \theta\otimes {\rm id}_{V})(w) 
	$$is called a {\it Mori-Smith twist} ({\it MS twist}) of $w$ by $\theta$. 
\end{defin}
For a potential $w \in V^{\otimes 3}$, 
we define
$$
{\rm Aut}\,(w):=
\{\theta \in {\rm GL}(V) \mid  
(\theta^{\otimes 3})(w)
=\lambda w,\,\exists \lambda \in k^{\times}\}.
$$
\begin{lem}
	[{\cite[Proposition 5.2]{MS1}}]
	\label{twist}
	For a superpotential $w\in V^{\otimes 3}$ and 
	$\phi \in {\rm GrAut}_{k}\,\mathcal{D}(w)$,
	we have that $\phi|_{V} \in \Aut\,(w)$ and
	$\mathcal{D}(w^{\phi|_{V}}) \cong \mathcal{D}(w)^{\phi}$.
\end{lem}
We say that a potential $w \in V^{\otimes 3}$ is {\it regular {\rm (}resp. Calabi-Yau{\rm )}}
if the derivation-quotient algebra $\mathcal{D}(w)$ is
a $3$-dimensional quadratic AS-regular (resp. Calabi-Yau AS-regular) algebra.
\begin{lem}[{\cite[Corollary 4.5]{MS1}}]\label{CY-superpotential}
	Let $w \in V^{\otimes 3}$ be a regular potential.
	Then $w$ is Calabi-Yau if and only if it is a superpotential.
\end{lem}
\begin{lem}[{\cite[Lemma 2.16]{IM2}}]
	\label{lem-1}
	If $w \in V^{\otimes 3}$ is a  regular superpotential and $\theta \in
	{\rm Aut}\,(w)
	$, 
	then the MS twist $w^{\theta}$
	is a regular twisted superpotential. 
\end{lem}
\begin{rem}
	If $\theta \in {\rm GL}(V)\setminus{\rm Aut}\,(w)$, 
	then the MS twist $w^{\theta}$ of a regular superpotential $w\in V^{\otimes 3}$ 
	by $\theta$ may not be regular nor a twisted superpotential
	(see \cite[Remark 2.12]{IM2}).
\end{rem}
\subsection{Elliptic curves in $\mathbb{P}^{2}$}
For the rest of this paper, let $E$ be an elliptic curve in $\mathbb{P}^{2}$
and we use a {\it Hesse form}
$E=\mathcal{V}(x^{3}+y^{3}+z^{3}-3\la xyz)$
where $\la \in k$ with $\la^{3} \neq 1$.
Note that every elliptic curve in $\mathbb{P}^{2}$ can be written in this form
(\cite[Corollary 2.18]{F}).
Our aim in this section is to recall some properties of
elliptic curves in $\mathbb{P}^{2}$.

The {\it $j$-invariant} of a Hesse form is given by the following formula
(see \cite[Proposition 2.16]{F}): 
$$
j(E)=\frac{27\la^{3}(\la^{3}+8)^{3}}{(\la^{3}-1)^{3}}. 
$$
It is well-known that
the $j$-invariant $j(E)$ classifies elliptic curves in $\mathbb{P}^{2}$
up to projective equivalence (\cite[Theorem IV 4.1 (b)]{H}).
For $o_{E}:=(1:-1:0) \in E$,
the group structure on $(E,o_{E})$ is given as follows (\cite[Theorem 2.11]{F}):
for $p=(a:b:c)$, $q=(\al:\be:\ga) \in E$,
$$p+q=\begin{cases}
&(ac\be^{2}-b^{2}\al\ga:bc\al^{2}-a^{2}\be\ga:ab\ga^{2}-c^{2}\al\be) \,\,\,\text{or, }\\
&(ab\al^{2}-c^{2}\be\ga:ac\ga^{2}-b^{2}\al\be:bc\be^{2}-a^{2}\al\ga).
\end{cases}$$
Throughout this paper, we fix the group structure on $E$ with the zero element
$o_{E}:=(1:-1:0) \in E$.

We define 
\begin{align*}
	\Aut_{k}(\mathbb{P}^{2},E)&:=\{ \sigma \in \Aut_{k}\,\mathbb{P}^{2} \mid \sigma|_{E}
	\in \Aut_{k}\,E \}, \\
	\Aut_{k}(E,o_{E})&:=
	\{
	\sigma \in \Aut_{k}\,E \mid \sigma(o_{E})=o_{E}
	\}.
\end{align*}
It is well-known that $\Aut_{k}(E,o_{E})$ is a finite cyclic subgroup of $\Aut_{k}\,E$
(\cite[Corollary IV 4.7]{H}).
A generator of $\Aut_{k}(E,o_{E})$ is given in the next lemma.
\begin{lem}[{\cite[Theorem 4.6]{IM1}}]\label{generat}
	A generator $\tau$ of $\Aut_{k}(E,o_{E})$ is given by
	\begin{enumerate}[{\rm (i)}]
		\item $\tau_{1}(a:b:c):=(b:a:c)$ if $j(E) \neq 0,12^{3}$,
		
		\item $\tau_{2}(a:b:c):=(b:a:c \ep)$ if $\la=0$ 
		{\rm (}so that $j(E)=0${\rm )},
		
		\item $\tau_{3}(a:b:c):=(a\ep^{2}+b\ep+c:a\ep+b\ep^{2}+c:a+b+c)$
		if $\la=1+\sqrt{3}$ {\rm (}so that $j(E)=12^{3}${\rm )}
	\end{enumerate}
	where $\ep$ is a primitive $3$rd root of unity.
	In particular, $\Aut_{k}(E,o_{E})$ is a subgroup of 
	${\rm Aut}_{k}(\mathbb{P}^{2},E)$.
\end{lem}
Note that $|\tau_{1}|=2$, $|\tau_{2}|=6$ and $|\tau_{3}|=4$.
Since $\tau_{2}^{3}=\tau_{3}^{2}=\tau_{1}$,
we use the notation $\tau_{1} \in \Aut_{k}\,(\mathbb{P}^{2},E)$ even when $j(E)=0,12^{3}$.
\begin{rem}\label{fx}
    When $j(E)=0,12^{3}$, we fix $\la=0, 1+\sqrt{3}$ respectively as in Lemma \ref{generat}
	to classify $\mathcal{A}(E,\sigma)$ up to graded algebra isomorphism,
	because if $E$ and $E'$ are projectively equivalent, then,
	for any $\mathcal{A}(E,\sigma)$,
	there is an automorphism $\sigma' \in \Aut_{k}\,E'$ such that 
	$\mathcal{A}(E,\sigma) \cong \mathcal{A}(E', \sigma')$ (see \cite[Lemma 2.6]{MU}).
\end{rem}
Every automorphism $\sigma \in {\rm Aut}_{k}\,E$
can be written as $\sigma=\sigma_{p}\tau^{i}$
where $\sigma_{p}$ is a translation by a point $p \in E$, 
$\tau$ is a generator of ${\rm Aut}_{k}(E,o_{E})$
and $i \in \mathbb{Z}_{|\tau|}$ (\cite[Proposition 4.5]{IM1}).

For any $n \ge 1$,
a point $p \in E$ is called {\it $n$-torsion} if $$np:=\underbrace{p + \cdots + p}_{n}=o_{E}.$$
We set $E[n]:=\{ p \in E \mid np=o_{E} \}$.
Note that $E[n]$ is a subgroup of $E$ and
$|E[n]|=n^{2}$ (\cite[Example 4.8.1]{H}).

We use the following easy lemma several times.
\begin{lem}\label{lem1}
	Let $E=\mathcal{V}(x^{3}+y^{3}+z^{3}-3 \la xyz)$ be an elliptic curve in $\mathbb{P}^{2}$
	where $\la \in k$ with $\la^{3} \neq 1$.
	\begin{enumerate}[{\rm (1)}]
		\item $E[6]=E[2] \oplus E[3]$.
		
		\item For $p=(a:b:c) \in E$, the inverse element of $p$ is
		$\tau_{1}(p)=(b:a:c)$.
		
		\item For $p=(a:b:c) \in E$, $p \in E[3]$ if and only if $abc=0$. 
		
		\item For $p=(a:b:c) \in E \setminus E[3]$,
		$p \in E[2]$ if and only if $a=b$.
	\end{enumerate}
\end{lem}
\section{Geometric pairs for Type EC algebras}\label{gomesec}
We say that a geometric algebra $\mathcal{A}(E,\sigma)$ is {\it of Type EC}
if $E$ is an elliptic curve in $\mathbb{P}^{2}$.
For a quadratic algebra $\mathcal{A}(E,\sigma_{p}\tau^{i})$ 
where $p \in E$, $\langle \tau \rangle=\Aut_{k}(\mathbb{P}^{2},E)$ and $i \in \mathbb{Z}_{|\tau|}$,
$\mathcal{A}(E,\sigma_{p}\tau^{i})$ is a geometric algebra of Type EC
if and only if $p \in E \setminus E[3]$ (\cite[Lemma 4.14]{IM1}).
In the previous paper \cite{IM1},
we gave a simple condition when two geometric algebras of Type EC are isomorphic.
For each $i \in \mathbb{Z}_{|\tau|}$, we define
$$U^{\tau^{i}}:=\{ r-\tau^{i}(r) \mid r \in E[3] \}.$$
\begin{thm}[{\cite[Theorem 4.16]{IM1}}]\label{iso}
	Let $E$ be an elliptic curve in $\mathbb{P}^{2}$, $p,q \in E \setminus E[3]$ 
	and $i,j \in \mathbb{Z}_{|\tau|}$.
	Then $\mathcal{A}(E,\sigma_{p}\tau^{i}) \cong \mathcal{A}(E,\sigma_{q}\tau^{j})$
	if and only if $i=j$ and $q=\tau^{l}(p)+r$ 
	where $r \in U^{\tau^{i}}$ and $l \in \mathbb{Z}_{|\tau|}$.
\end{thm}
\begin{lem}[{\cite[Lemma 4.17]{IM1}}]\label{points}
	$$U^{\tau^{i}}=\begin{cases}
	\{o_{E}\} \,\,\,&\text{if } i=0, \\
	\langle (1:-\ep:0) \rangle \,\,\, &\text{if } \la=0 \text{ and } i=2,4, \\
	E[3] \,\,\,&\text{otherwise}
	\end{cases}$$
	where $\ep$ is a primitive $3$rd root of unity.
\end{lem}
In the previous paper \cite{IM2}, we gave a simple condition
to determine whether a geometric algebra of Type EC is AS-regular.
For each $i \in \mathbb{Z}_{|\tau|}$,
we define 
$$U_{\tau^{i}}:=\{ p \in E \mid p-\tau^{i}(p) \in E[3] \}.$$
\begin{thm}[{\cite[Theorem 4.3]{IM2}}]\label{char}
For $p \in E$ and $i \in \mathbb{Z}_{|\tau|}$,
	$\mathcal{A}(E,\sigma_{p}\tau^{i})$ is a $3$-dimensional quadratic AS-regular algebra
	if and only if $p \in U_{\tau^{i}} \setminus E[3]$.
\end{thm}
Since $U_{\tau^{0}}=E$,
$\mathcal{A}(E,\sigma_{p})$
is AS-regular if and only if $p \in E \setminus E[3]$.
A $3$-dimensional quadratic AS-regular algebra $\mathcal{A}(E,\sigma_{p})$
is called a {\it $3$-dimensional Sklyanin algebra}.
Since $\tau$ is a group homomorphism of $E$,
$E[3] \subset U_{\tau^{i}}$
for any $i \in \mathbb{Z}_{|\tau|}$.


For each $i \in \mathbb{Z}_{|\tau|}$, we define
$$E_{\tau^{i}}:=\{ p \in E \mid \tau^{i}(p)=p \}.$$
It is easy to check the following lemma.
\begin{lem}\label{fixed}
	\begin{enumerate}[{\rm (1)}]
		\item $E_{\tau^{0}}=E$.
		\item$E_{\tau_{1}}=E[2]$.
		\item For any $i \in \mathbb{Z}_{|\tau|}$, $E_{\tau^{i}}=E_{\tau^{|\tau|-i}}$.
	\end{enumerate}
\end{lem}
The next lemma is important to determine $U_{\tau^{i}}$.
\begin{lem}\label{fix}
	For any $i \in \mathbb{Z}_{|\tau|}$,
	$$U_{\tau^{i}}=\{ p \in E \mid 3p \in E_{\tau^{i}} \}.$$
\end{lem}
\begin{proof}
	For any $i \in \mathbb{Z}_{|\tau|}$,
	$p \in U_{\tau^{i}}$ if and only if $p-\tau^{i}(p) \in E[3]$.
	Since $\tau$ is a group homomorphism,
	$p-\tau^{i}(p) \in E[3]$ if and only if $\tau^{i}(3p)=3p$.
	Hence $p \in U_{\tau^{i}}$ if and only if $3p \in E_{\tau^{i}}$.
\end{proof}
\begin{prop}\label{prop01}
	\begin{enumerate}[{\rm (1)}]
		\item If $j(E) \neq 0,12^{3}$, then
		$$E_{\tau_{1}^{i}}=\begin{cases}
			E \,\,\,& \text{if } i=0, \\
			E[2] \,\,\,& \text{if } i=1.
		\end{cases}$$
		
		\item If $\la=0$, then
		$$E_{\tau_{2}^{i}}=\begin{cases}
		E \,\,\,& \text{if } i=0, \\
		\{ o_{E}\} \,\,\,& \text{if }i=1,5, \\
		\langle (1:-\ep:0) \rangle \,\,\,& \text{if }i=2,4, \\
		E[2] \,\,\,& \text{if } i=3.
		\end{cases}$$
		where $\ep$ is a primitive $3$rd root of unity.
		
		\item If $\la=1+\sqrt{3}$, then
		$$E_{\tau_{3}^{i}}=\begin{cases}
		E \,\,\,& \text{if } i=0, \\
		\langle (1:1:\la) \rangle \,\,\,& \text{if } i=1,3, \\
		E[2] \,\,\,& \text{if } i=2.
		\end{cases}$$
	\end{enumerate}
\end{prop}
\begin{proof}
	\begin{enumerate}[{\rm (1)}]
		\item By Lemma \ref{fixed}, the formulas hold.
		
		\item Assume that $\la=0$. 
		By Lemma \ref{fixed}, it is sufficient to prove the cases $i=1,2$.
		\begin{enumerate}[{\rm (i)}]
			\item If $i=1$, then $p=(a:b:c) \in E_{\tau_{2}}$ if and only if $(b:a:c\ep)=(a:b:c)$.
			If $c=0$, then there is $\alpha \in k^{\times}$ such that $a=\alpha b$ and $b=\alpha a$,
			so we have that $p=(1:1:0)$ or $p=(1:-1:0)=o_{E}$.
			Since $(1:1:0) \notin E$, it follows that $p=o_{E}$.
			If $c \neq 0$, then $a=\ep b$ and $b=\ep a$.
			Since $\ep$ is a primitive $3$rd root of unity, we have that $a=b=0$
			but $(0:0:1) \notin E$. Hence $E_{\tau_{2}}=\{ o_{E} \}$.
			
			\item  If $i=2$, then $p=(a:b:c) \in E_{\tau_{2}^{2}}$ if and only if $(a:b:c\ep^{2})=(a:b:c)$.
			If $c \neq 0$, then since $a=\ep^{2}a$ and $b=\ep^{2}b$, 
			it follows that $a=b=0$ but $(0:0:1) \notin E$.
			If $c=0$, then $a^{3}+b^{3}=0$, so $p=o_{E}, (1:-\ep:0), (1:-\ep^{2}:0)$.
			Hence $E_{\tau_{2}^{2}}=\langle (1:-\ep:0) \rangle$.
		\end{enumerate}
	
	\item Assume that $\la=1+\sqrt{3}$. 
	By Lemma \ref{fixed}, it is sufficient to prove the case $i=1$.
	If $p=(a:b:c) \in E_{\tau_{3}}$, then since $\tau_{3}^{2}(p)=\tau_{3}(p)=p$,
	it follows that $p \in E_{\tau_{3}^{2}}=E[2]$ by Lemma \ref{fixed}.
	Let $p \in E_{\tau_{3}} \setminus \{o_{E}\}$.
	By Lemma \ref{fixed}, $a=b$, so we assume that $p=(1:1:c)$.
	Since $\tau_{3}(p)=(\ep^{2}+\ep+c:\ep+\ep^{2}+c:1+1+c)=(1:1:c)$, there is $\alpha \in k^{\times}$ such that
	$c-1=\alpha$ and $c+2=\alpha c$. Then we have that
	$$c^{2}-2c-2=0.$$
	By solving this equation, we have that $c=1\pm \sqrt{3}$.
	Since $(1:1:1+\sqrt{3}) \in E$ but $(1:1:1-\sqrt{3}) \notin E$,
	$E_{\tau_{3}} \leq \langle (1:1:1+\sqrt{3}) \rangle$.
	Conversely,
	$$\tau_{3}(1:1:1+\sqrt{3})=(\sqrt{3}:\sqrt{3}:3+\sqrt{3})=(1:1:1+\sqrt{3}),$$
	so $\langle (1:1:1+\sqrt{3}) \rangle \leq E_{\tau_{3}}$. 
	Hence $E_{\tau_{3}}=\langle (1:1:1+\sqrt{3}) \rangle$.
	\end{enumerate}
\end{proof}
\begin{lem}\label{two}
	If $E_{\tau^{i}} \leq E[2]$, then $U_{\tau^{i}}=E_{\tau^{i}} \oplus E[3]$.
\end{lem}
\begin{proof}
	Assume that $E_{\tau^{i}} \leq E[2]$. Since $E[2] \cap E[3]=\{o_{E}\}$, $E_{\tau^{i}} \cap E[3]=\{o_{E}\}$.
	If $p \in U_{\tau^{i}}$, then $3p \in E_{\tau^{i}} \leq E[2]$ by Lemma \ref{fix},
	so $U_{\tau^{i}} \leq E[6]=E[2] \oplus E[3]$.
	For any $p \in U_{\tau^{i}}$, there are $q \in E[2]$ and $r \in E[3]$ such that $p=q+r$.
	Since $3p \in E_{\tau^{i}}$,
	$$q=q+o_{E}=3q+3r=3p \in E_{\tau^{i}},$$
	so we have that $U_{\tau^{i}} \leq E_{\tau^{i}} \oplus E[3]$.
	By definition of $E_{\tau^{i}}$, we have that $E_{\tau^{i}} \leq U_{\tau^{i}}$,
	so $E_{\tau^{i}} \oplus E[3] \leq U_{\tau^{i}}$.
	Hence $U_{\tau^{i}}=E_{\tau^{i}} \oplus E[3]$.
\end{proof}
\begin{thm}\label{main01}
	\begin{enumerate}[{\rm (1)}]
		\item If $j(E) \neq 0,12^{3}$, then
		$$U_{\tau_{1}^{i}}=\begin{cases}
		E \,\,\,& \text{if } i=0, \\
		E[6] \,\,\,& \text{if } i=1.
		\end{cases}$$
		
		\item If $\la=0$, then
		$$U_{\tau_{2}^{i}}=\begin{cases}
		E \,\,\,& \text{if } i=0, \\
		E[3] \,\,\,& \text{if } i=1,5, \\
		E[3] \sqcup \{ \tau_{2}^{l}(q)+r \mid l \in \mathbb{Z}_{6}, r \in E_{\tau_{2}^{2}} \} \,\,\,& \text{if } i=2,4, \\
		E[6] \,\,\,& \text{if } i=3,
		\end{cases}$$
		where $q=(\eta^{8}:\eta^{4}:1)$ and $\eta$ is a primitive $9$th root of unity.
		
		\item If $\la=1+\sqrt{3}$, then 
		$$U_{\tau_{3}^{i}}=\begin{cases}
		E \,\,\,& \text{if } i=0, \\
		\langle (1:1:\la) \rangle \oplus E[3] \,\,\,& \text{if } i=1,3, \\
		E[6] \,\,\,& \text{if } i=2.
		\end{cases}$$
	\end{enumerate}
\end{thm}
\begin{proof}
	For each case, if $i=0$, then $U_{\tau^{0}}=E$.
	\begin{enumerate}[{\rm (1)}]
		\item
		If $i=1$, then 
		$U_{\tau_{1}}=E[2] \oplus E[3]=E[6]$
		by Proposition \ref{prop01} (1) and Lemma \ref{two}.
		
		\item Assume that $\la=0$.
		Since $\tau_{2}^{3}=\tau_{1}$,
		it follows that $U_{\tau_{2}^{3}}=E[6]$.
		Since $U_{\tau_{2}^{i}}=U_{\tau_{2}^{6-i}}$, it is sufficient to prove the cases $i=1,2$.
		
		If $i=1$, then 
		$U_{\tau_{2}}=E[3]$
		by Proposition \ref{prop01} (2) and Lemma \ref{two}.
		
        Assume that $i=2$.
		Let $p=(a:b:c) \in U_{\tau_{2}^{2}} \setminus E[3]$.
		Since $abc \neq 0$, we may assume that $c=1$.
		Note that $p=(a:b:1)$ satisfies $a^{3}+b^{3}+1=0$ and $ab \neq 0$.
		Since $3p \in E_{\tau^{2}}$, the third component of $3p$ is equal to $0$, that is,
		$$
		ab(b^{3}-a^{3})^{2}-ab(a^{3}-1)(1-b^{3})=0.
		$$
		Since $b^{3}=-(a^{3}+1)$ and $ab \neq 0$, we have that
		$a^{6}+a^{3}+1=0$.
		Since $k$ is an algebraically closed field, the set of solutions is
		$\{ \eta,\eta^{2},\eta^{4},\eta^{5},\eta^{7},\eta^{8} \}$
		where $\eta$ is a primitive $9$th root of unity.
		Then there are two cases $(1)\,\,a^{3}=\eta^{3}$ and $(2)\,\,a^{3}=\eta^{6}$. 
		If $a^{3}=\eta^{3}$, then $b^{3}=\eta^{6}$, so $b=\eta^{2}, \eta^{5}, \eta^{8}$ and
		if $a^{3}=\eta^{6}$, then $b^{3}=\eta^{3}$, so $b=\eta, \eta^{4}, \eta^{7}$.
		Consequently, we have that
		$$U_{\tau_{2}^{2}} \subset E[3] \sqcup \{ (a:b:1) \mid (a^{3},b^{3})=(\eta^{3},\eta^{6}) \} \sqcup
		\{ (a:b:1) \mid (a^{3},b^{3})=(\eta^{6},\eta^{3}) \}.$$
		Conversely, if $p \in  \{ (a:b:1) \mid (a^{3},b^{3})=(\eta^{3},\eta^{6}) \}$,
		then $$2p=(b:a\eta^{6}:\eta^{3}) \text{ and } 3p=(1:-\eta^{6}:0) \in E_{\tau_{2}^{2}}$$
		by Proposition \ref{prop01}, so $p \in U_{\tau_{2}^{2}}$ by Lemma \ref{fix}.
		Similarly, if $$p \in \{ (a:b:1) \mid (a^{3},b^{3})=(\eta^{6},\eta^{3}) \},$$ then
		$$2p=(b\eta^{6}:a:\eta^{3}) \text{ and } 3p=(1:-\eta^{3}:0) \in E_{\tau_{2}^{2}},$$
		so $p \in U_{\tau_{2}^{2}}$.
		Thus $$U_{2}=E[3] \sqcup \{ (a:b:1) \mid (a^{3},b^{3})=(\eta^{3},\eta^{6}) \} \sqcup
		\{ (a:b:1) \mid (a^{3},b^{3})=(\eta^{6},\eta^{3}) \}.$$ 
			Since 
		\begin{align*}
		&\{ \tau_{2}^{l}(q)+r \mid l \in \mathbb{Z}_{6}, r \in E_{\tau_{2}^{2}} \} \\
		=\{
		&(\eta^{8}:\eta^{4}:1),\, (\eta^{2}:\eta:1),\, (\eta^{5}:\eta^{7}:1),\, (\eta^{5}:\eta:1),\, 
		(\eta^{2}:\eta^{4}:1),\, (\eta^{8}:\eta^{7}:1),\,  \\
		&(\eta^{5}:\eta^{4}:1),\, (\eta^{8}:\eta:1),\, 
		(\eta^{2}:\eta^{7}:1),\, (\eta:\eta^{5}:1),\, (\eta^{4}:\eta^{2}:1),\, (\eta^{7}:\eta^{8}:1),\, \\
		&(\eta^{4}:\eta^{8}:1),\, (\eta:\eta^{2}:1),\, (\eta^{7}:\eta^{5}:1),\, (\eta^{4}:\eta^{5}:1),\, 
		(\eta^{7}:\eta^{2}:1),\, (\eta:\eta^{8}:1) \},
		\end{align*}
		we have that 
		$U_{\tau_{2}^{2}}=E[3] \sqcup \{ \tau_{2}^{l}(q)+r \mid l \in \mathbb{Z}_{6}, r \in E_{\tau_{2}^{2}} \}$.
		
		\item Assume that $\la=1+\sqrt{3}$.
		Since $\tau_{3}^{2}=\tau_{1}$, it follows that $U_{\tau_{3}^{2}}=E[6]$.
		Since $U_{\tau_{3}^{i}}=U_{\tau_{3}^{4-i}}$, it is sufficient to prove the case $i=1$.
		
		If $i=1$, then 
		$U_{\tau_{3}}=\langle (1:1:\la) \rangle \oplus E[3]$
		by Proposition \ref{prop01} (3) and Lemma \ref{two}.
	\end{enumerate}

\end{proof}
\begin{rem}
	By Theorem \ref{main01}, if $\la=0$ and $i=1,5$, then a geometric algebra $\mathcal{A}(E,\sigma_{p}\tau_{2}^{i})$
	of Type EC is never AS-regular.
\end{rem}

We define the four types of $3$-dimensional quadratic AS-regular algebras $\mathcal{A}(E,\sigma)$ of Type EC
mimicking \cite[4.13]{ATV1}:
\begin{description}
	\item [{\rm (1) Type A}]
	$\sigma=\sigma_{p}$.
	
	\item [{\rm (2) Type B}]
	$\sigma=\sigma_{p}\tau_{1}$.
	
	\item [{\rm (3) Type E}]
	$j(E)=0$ and $\sigma=\sigma_{p}\tau_{2}^{2}, \sigma_{p}\tau_{2}^{4}$.
	
	\item [{\rm (4) Type H}]
	$j(E)=12^{3}$ and $\sigma=\sigma_{p}\tau_{3}, \sigma_{p}\tau_{3}^{3}$.
\end{description}
By Theorem \ref{iso},
if $X \neq Y$, then Type X algebra is not isomorphic to any Type Y algebra.
\begin{prop}\label{01}
	Fix an elliptic curve $E$.
	\begin{enumerate}[{\rm (1)}]
		\item If $j(E) \neq 0,12^{3}$, then
		there are three Type B algebras up to isomorphism.
		
		\item If $j(E)=0$, then there is one Type B algebra up to isomorphism.
		
		\item If $j(E)=12^{3}$, then there are two Type B algebras up to isomorphism.
	\end{enumerate}
\end{prop}

\begin{proof}
	Let $\mathcal{A}(E,\sigma_{p}\tau_{1})$ be a Type B algebra
	where $p \in U_{\tau_{1}} \setminus E[3]$.
	By Theorem \ref{main01}, there are $q \in E[2]$ and $r \in E[3]$
	such that $p=q+r$. 
	Since $U^{\tau_{1}}=E[3]$ by Lemma \ref{points},
	it follows from Theorem \ref{iso} that $\mathcal{A}(E,\sigma_{p}\tau_{1})
	\cong \mathcal{A}(E,\sigma_{q}\tau_{1})$,
	so we may assume that $p \in E[2] \setminus \{o_{E}\}$.
	\begin{enumerate}[{\rm (1)}]
		\item Assume that $j(E) \neq 0,12^{3}$.
		Let $\mathcal{A}(E,\sigma_{p}\tau_{1})$ and $\mathcal{A}(E,\sigma_{q}\tau_{1})$
		be Type B algebras
		where $p,q \in E[2] \setminus \{o_{E}\}$.
		By Theorem \ref{iso},
		$\mathcal{A}(E,\sigma_{p}\tau_{1}) \cong \mathcal{A}(E,\sigma_{q}\tau_{1})$ 
		if and only if there exist $r \in U^{\tau_{1}}$ and $l \in \mathbb{Z}_{2}$
		such that $q=\tau_{1}^{l}(p)+r$.
		Since $U^{\tau_{1}}=E[3]$ by Lemma \ref{points} and $\tau_{1}(p)=p$,
		we have that
		$\mathcal{A}(E,\sigma_{p}\tau_{1}) \cong \mathcal{A}(E,\sigma_{p}\tau_{1})$ if and only if $q-p \in E[3]$.
		Since $E[2] \cap E[3]=\{o_{E}\}$,
		$\mathcal{A}(E,\sigma_{p}\tau_{1}) \cong \mathcal{A}(E,\sigma_{q}\tau_{1})$ if and only if $p=q$.
		It follows from $|E[2]|=2^{2}=4$ that there are three Type B algebras up to isomorphism.
		
		\item Assume that $\la=0$.
		Let $p=(1:1:c) \in E[2] \setminus \{o_{E}\}$.
		Since $E=\mathcal{V}(x^{3}+y^{3}+z^{3})$,
		we have that $c^{3}+2=0$, so
		$$c=-\sqrt[3]{2},-\sqrt[3]{2}\ep,-\sqrt[3]{2}\ep^{2}$$
		where $\ep$ is a primitive $3$rd root of unity.
		We set $p_{1}=(1:1:-\sqrt[3]{2}), p_{2}=(1:1:-\sqrt[3]{2}\ep), p_{3}=(1:1:-\sqrt[3]{2}\ep^{2})$.
		Since $p_{2}=\tau_{2}^{4}(p_{1})$ and $p_{3}=\tau_{2}^{2}(p_{1})$,
		we have that
		$\mathcal{A}(E,\sigma_{p_{1}}\tau_{2}^{3}) \cong \mathcal{A}(E,\sigma_{p_{2}}\tau_{2}^{3})
		\cong \mathcal{A}(E,\sigma_{p_{3}}\tau_{2}^{3})$ by Theorem \ref{iso}.
		Thus there is one Type B algebra up to isomorphism.
		
		\item Assume that $\la=1+\sqrt{3}$.
		Let $p_{1}=(1:1:\la)$. Since $E_{\tau_{3}} = \langle p_{1} \rangle$ by Proposition \ref{prop01}
		and $U^{\tau_{3}^{2}}=E[3]$ by Lemma \ref{points},
		$$\{ \tau_{3}^{l}(p_{1})+r \mid l \in \mathbb{Z}_{4}, r \in U^{\tau_{3}^{2}} \}
		=\{ p_{1}+r \mid r \in E[3] \},$$
		so if $p \in E[2] \setminus \{o_{E},p_{1}\}$, then
		$\mathcal{A}(E,\sigma_{p}\tau_{3}^{2}) \not\cong \mathcal{A}(E,\sigma_{p_{1}}\tau_{3}^{2})$.
		Let $p,q \in E[2] \setminus \{o_{E},p_{1}\}$ and assume that $p \neq q$.
		Since $E_{\tau_{3}} = \langle p_{1} \rangle$ by Proposition \ref{prop01},
		$\tau_{3}(p)=q$, so it follows from Theorem \ref{iso} that
		$\mathcal{A}(E,\sigma_{p}\tau_{3}^{2}) \cong \mathcal{A}(E,\sigma_{q}\tau_{3}^{2})$.
		Hence there are two Type B algebras up to isomorphism.
	\end{enumerate}
\end{proof}

\begin{prop}\label{5-1-1}
	There are two Type E algebras up to isomorphism.
	   
	   
\end{prop}

\begin{proof}
	Assume that $\la=0$.
	Let $q=(\eta^{8}:\eta^{4}:1)$
	where $\eta$ is a primitive $9$th root of unity.
	\begin{enumerate}[{\rm (1)}]
		\item Let $\mathcal{A}(E,\sigma_{p}\tau_{2}^{2})$ be a Type E algebra
		where $p \in U_{\tau_{2}^{2}} \setminus E[3]$. 
		By Theorem \ref{main01}, there are $l \in \mathbb{Z}_{6}$ and $r \in E_{\tau_{2}^{2}}$
		such that $p=\tau_{2}^{l}(q)+r$.
		By Lemma \ref{points} and Proposition \ref{prop01},
		$U^{\tau_{2}^{2}}=\langle (1:-\eta^{3}:0) \rangle =E_{\tau_{2}^{2}}$,
		so it follows from Theorem \ref{iso} that
		$\mathcal{A}(E,\sigma_{p}\tau_{2}^{2}) \cong \mathcal{A}(E,\sigma_{q}\tau_{2}^{2})$.
		
		
		\item Let $\mathcal{A}(E,\sigma_{p}\tau_{2}^{4})$ be a Type E algebra
		where $p \in U_{\tau_{2}^{4}} \setminus E[3]$. 
		By Theorem \ref{main01}, there are $l \in \mathbb{Z}_{6}$ and $r \in E_{\tau_{2}^{4}}$
		such that $p=\tau_{2}^{l}(q)+r$.
		By Lemma \ref{points} and Lemma \ref{fixed},
		$U^{\tau_{2}^{4}}=\langle (1:-\eta^{3}:0) \rangle=E_{\tau_{2}^{4}}$,
		so it follows from Theorem \ref{iso} that
		$\mathcal{A}(E,\sigma_{p}\tau_{2}^{4}) \cong \mathcal{A}(E,\sigma_{q}\tau_{2}^{4})$.
	\end{enumerate}
	By Theorem \ref{iso}, $\mathcal{A}(E,\sigma_{q}\tau_{2}^{2}) \not\cong \mathcal{A}(E,\sigma_{q}\tau_{2}^{4})$,
	so there are two Type E algebras up to isomorphism.
\end{proof}

\begin{prop}\label{prop04}
	There are two Type H algebras up to isomorphism.
		
		
		
\end{prop}

\begin{proof}
	Assume that $\la=1+\sqrt{3}$.
	Let $q=(1:1:\la)$.
	\begin{enumerate}[{\rm (1)}]
		\item Let $\mathcal{A}(E,\sigma_{p}\tau_{3})$ be a Type H algebra where
		$p \in U_{\tau_{3}} \setminus E[3]$.
		By Theorem \ref{main01}, there exists $r \in E[3]$ such that
		$p=q+r$.
		By Lemma \ref{points}, $U^{\tau_{3}}=E[3]$, so
		it follows from Theorem \ref{iso} that
		$
		\mathcal{A}(E,\sigma_{p}\tau_{3}) \cong \mathcal{A}(E,\sigma_{q}\tau_{3})
		$.
		
		
		\item Let $\mathcal{A}(E,\sigma_{p}\tau_{3}^{3})$ be a Type H algebra where
		$p \in U_{\tau_{3}^{3}} \setminus E[3]$.
		By Theorem \ref{main01}, there exists $r \in E[3]$ such that
		$p=q+r$.
		By Lemma \ref{points}, $U^{\tau_{3}^{3}}=E[3]$, so
		it follows from Theorem \ref{iso} that
		$
		\mathcal{A}(E,\sigma_{p}\tau_{3}^{3}) \cong \mathcal{A}(E,\sigma_{q}\tau_{3}^{3})
		$.
		
	\end{enumerate}
By Theorem \ref{iso}, $\mathcal{A}(E,\sigma_{q}\tau_{3}) \not\cong \mathcal{A}(E,\sigma_{q}\tau_{3}^{3})$,
so there are two Type H algebras up to isomorphism.
\end{proof}
We call a geometric pair $(E,\sigma)$ {\it regular} 
if $\mathcal{A}(E,\sigma)$ is a
$3$-dimensional quadratic AS-regular 
algebra. 
The following table gives a complete list of regular geometric pairs
for Type EC algebras up to isomorphism.
\begin{thm}\label{main03}
	\begin{enumerate}[{\rm (1)}]
		\item Every Type EC algebra is isomorphic to $\mathcal{A}(E,\sigma)$
		where $(E,\sigma)$ is in Table $1$.
		\item Every $(E,\sigma)$ in Table $1$ is regular.
	\end{enumerate}
\begin{center}
	Table 1 (A list of regular geometric pairs)
	
	\begin{tabular}{|c||c|c|c|}
		\hline
		Type & \multicolumn{3}{c|}{ regular geometric pairs $(E,\sigma)$} \\ \hline
		%
		\multirow2*{$A$}
		&
		\multicolumn{3}{c|}{$\sigma=\sigma_{p}$}\\ 
		&
		\multicolumn{3}{c|}{where $p=(a:b:c)$ satisfies $abc \neq 0$ and $(a^{3}+b^{3}+c^{3})^{3} \neq (3abc)^{3}$}\\ 
		\hline
		%
		\multirow2*{$B$}
		&
		\multicolumn{3}{c|}{$\sigma=\sigma_{p}\tau_{1}$}\\ 
		&
		\multicolumn{3}{c|}{where $p=(1:1:c)$ satisfies $c^{3}-3\la c+2=0$}\\ 
		\hline
		%
		\multirow2*{$E$}
		&
		\multicolumn{3}{c|}{$\la=0$, \,\,\,$\sigma=\sigma_{p}\tau_{2}^{2}, \sigma_{p}\tau_{2}^{4}$}\\ 
		&
		\multicolumn{3}{c|}{where $p=(\eta^{8}:\eta^{4}:1)$ and $\eta$ is a primitive $9$th root of unity}\\ 
		\hline
		\multirow{2}{*}{$H$}
		&
		\multicolumn{3}{c|}{$\la=1+\sqrt{3}$, \,\,\,
			$\sigma=\sigma_{p}\tau_{3}, \sigma_{p}\tau_{3}^{3}$
		}\\
		&
		\multicolumn{3}{c|}{
			where $p=(1:1:\la)$
		}\\ \hline
	\end{tabular}
\end{center}
\end{thm}
\begin{proof}
	\begin{enumerate}[{\rm (1)}]
		\item By \cite[Lemma 4.7]{IM1}, every Type EC algebra is isomorphic to
		$\mathcal{A}(E,\sigma_{p}\tau^{i})$ 
		where $p=(a:b:c) \in U_{\tau^{i}} \setminus E[3]$ and $i \in \mathbb{Z}_{|\tau|}$.
		If $i=0$, then $p \in E \setminus E[3]$ by Theorem \ref{main01},
		so
		$$abc \neq 0 \text{ and } (a^{3}+b^{3}+c^{3})^{3} \neq (3abc)^{3}.$$
		If $i \neq 0$, then, by the proofs of Propositions \ref{01}, \ref{5-1-1} and \ref{prop04},
		$\mathcal{A}(E,\sigma_{p}\tau^{i})$ is isomorphic to $\mathcal{A}(E,\sigma)$
		where $(E,\sigma)$ is in Table $1$.
		
		\item By Lemma \ref{char}, every $(E,\sigma)$ in Table $1$ is regular. 
	\end{enumerate}
\end{proof}
\section{Twisted superpotentials for Type EC algebras}\label{tspsec}
In this section, we will give a list of regular twisted superpotentials for Type EC algebras.
As a byproduct, we will prove that
a $3$-dimensional quadratic AS-regular algebra of Type EC is Calabi-Yau
if and only if it is a $3$-dimensional Sklyanin algebra.
	
	
	
For $p=(a:b:c) \in E \setminus E[3]$, we set
$$w_{p}:=a(xyz+yzx+zxy)+b(xzy+yxz+zyx)+c(x^{3}+y^{3}+z^{3}).$$
Since the derivation-quotient algebra $\mathcal{D}(w_{p})$
is a $3$-dimensional Sklyanin algebra $\mathcal{A}(E,\sigma_{p})$,
we call it a {\it Sklyanin potential} which is a regular superpotential.
The next lemma is important to determine reguar twisted superpotentials.

\begin{lem}\label{tsp}
      Let $w_{p}$ be a Sklyanin potential where $p \in E \setminus E[3]$.
      Assume that $\tau^{i} \in \Aut\,(w_{p})$ where $i \in \mathbb{Z}_{|\tau|}$.
      Then $(w_{p})^{\tau^{i}}$ is a regular twisted superpotential and
      $\mathcal{A}(E,\sigma_{p}\tau^{i}) \cong \mathcal{D}(w_{p})^{\tau^{i}}
      \cong \mathcal{D}((w_{p})^{\tau^{i}})$.
\end{lem}
\begin{proof}
      Since a Sklyanin potential $w_{p}$ is a regular superpotential and
      $\tau^{i} \in \Aut\,(w_{p})$ where $i \in \mathbb{Z}_{|\tau|}$,
      $(w_{p})^{\tau^{i}}$ is a regular twisted superpotential by Lemma \ref{lem-1}.
      By \cite[Theorem 4.9]{IM1} and Lemma \ref{twist},
      we have that
      $$\mathcal{A}(E,\sigma_{p}\tau^{i}) \cong \mathcal{D}(w_{p})^{\tau^{i}}
      \cong \mathcal{D}((w_{p})^{\tau^{i}}).$$
\end{proof}
The following table gives a complete list of regular twisted superpotentials
for Type EC algebras up to isomorphism.
\begin{thm}\label{main04}
	\begin{enumerate}[{\rm (1)}]
		\item Every Type EC algebra is isomorphic to $\mathcal{D}(w)$
		where $w$ is in Table $2$.
		\item Every $w$ in Table $2$ is a regular twisted superpotential.
	\end{enumerate}
\begin{center}
	Table 2 (A list of regular twisted superpotentials)\\
	\begin{tabular}{|c||c|c|c|}
		\hline
		Type & \multicolumn{3}{c|}{ regular twisted superpotetials $w$} \\ \hline
		%
		\multirow2*{$A$}
		&
		\multicolumn{3}{c|}{$a(xyz+yzx+zxy)+b(xzy+yxz+zyx)+c(x^{3}+y^{3}+z^{3})$}\\ 
		&
		\multicolumn{3}{c|}{where $abc \neq 0$ and $(a^{3}+b^{3}+c^{3})^{3} \neq (3abc)^{3}$}\\
		\hline
		%
		\multirow2*{$B$}
		&
		\multicolumn{3}{c|}{$(x^{z}+yzx+zy^{2})+(xzy+y^{2}z+zx^{2})+c(xyx+yxy+z^{3})$}\\ 
					&
					\multicolumn{3}{c|}{where $c^{3}-\la c+2=0$}\\ 
		\hline
		%
		\multirow6*{$E$}
		& \multicolumn{3}{c|}{$xzx+\eta zx^{2}+\eta^{8}x^{2}z+yxy+\eta^{4} xy^{2}+\eta^{5}y^{2}x$} \\
		& %
		\multicolumn{3}{c|}{$+zyz+\eta^{7}yz^{2}+\eta^{2}z^{2}y$} \\ 
		& %
		\multicolumn{3}{c|}{where $\eta$ is a primitive $9$th root of unity.} \\ \cline{2-4}
		&%
		\multicolumn{3}{c|}{$xzx+\eta^{8} zx^{2}+\eta x^{2}z+yxy+\eta^{5}xy^{2}+\eta^{4}y^{2}x$} \\
		& 
		\multicolumn{3}{c|}{$+zyz+\eta^{2}yz^{2}+\eta^{7}z^{2}y$}  \\
		& %
		\multicolumn{3}{c|}{where $\eta$ is a primitive $9$th root of unity.} \\ 
		\hline
		\multirow{12}{*}{$H$}
		& \multicolumn{3}{c|}{$(\ep xyz+\la yzx+\ep^{2}zxy)+(\la xzy+\ep yxz+\ep^{2}zyx)$} \\
		& \multicolumn{3}{c|}{$+(\ep\la x^{2}y+xyx+\ep^{2}\la yx^{2})+(\ep^{2}x^{2}z+xzx+\ep zx^{2})$} \\
		& \multicolumn{3}{c|}{$+(\ep \la y^{2}x+yxy+\ep^{2}\la yx^{2})+(\ep^{2}y^{2}z+yzy+\ep zy^{2})$} \\
		& \multicolumn{3}{c|}{$+(z^{2}x+\la zxz+xz^{2})+(z^{2}y+\la zyz+yz^{2})$} \\
		& \multicolumn{3}{c|}{$+(x^{3}+y^{3}+\la z^{3})$} \\ 
		& \multicolumn{3}{c|}{where $\ep$ is a primitive $3$rd root of unity and $\la=1+\sqrt{3}$}\\ \cline{2-4}
		& \multicolumn{3}{c|}{$(\ep^{2} xyz+\la yzx+\ep zxy)+(\la xzy+\ep^{2} yxz+\ep zyx)$} \\
		& \multicolumn{3}{c|}{$+(\ep^{2}\la x^{2}y+xyx+\ep\la yx^{2})+(\ep x^{2}z+xzx+\ep^{2} zx^{2})$} \\ 
		& \multicolumn{3}{c|}{$+(\ep^{2} \la y^{2}x+yxy+\ep\la yx^{2})+(\ep y^{2}z+yzy+\ep^{2} zy^{2})$} \\
		& \multicolumn{3}{c|}{$+(z^{2}x+\la zxz+xz^{2})+(z^{2}y+\la zyz+yz^{2})$} \\
		& \multicolumn{3}{c|}{$+(x^{3}+y^{3}+\la z^{3})$} \\
		&\multicolumn{3}{c|}{where $\ep$ is a primitive $3$rd root of unity and $\la=1+\sqrt{3}$} \\
		\hline
	\end{tabular}
\end{center}
\end{thm}
\begin{proof}
	To give a complete list of regular twisted superpotentials,
	it is enough to find a twisted superpotential $w$ such that
	$\mathcal{A}(E,\sigma) \cong \mathcal{D}(w)$
	for each $(E,\sigma)$ in Table $1$.
	If $\mathcal{A}(E,\sigma_{p}\tau^{i})$ is not of Type E,
	then we will show that $\tau^{i} \in \Aut\,(w_{p})$
	so that $(w_{p})^{\tau^{i}}$ is a regular twisted superpotential by Lemma \ref{tsp}
	and potentials in Table 2 are given by $(w_{p})^{\tau^{i}}$.
		
		\begin{enumerate}[{\rm (i)}]
			\item {\rm Type A}.
			Since a Sklyanin potential $w_{p}$ is a superpotential,
			the result follows.
			\item {\rm Type B}.
			Let $p=(1:1:c) \in E[2] \setminus \{o_{E}\}$.
			Since
			$(\tau_{1} \otimes \tau_{1} \otimes \tau_{1})(w_{p})
			=(yxz+xzy+zyx)+(yzx+xyz+zxy)+c(x^{3}+y^{3}+z^{3})=w_{p}$,
			$\tau_{1} \in \Aut\,(w_{p})$.
			By Lemma \ref{tsp},
			\begin{align*}
				(w_{p})^{\tau_{1}}&=({\rm id} \otimes \tau_{1} \otimes {\rm id})(w_{p}) \\
				&=(x^{2}z+yzx+zy^{2})+(xzy+y^{2}z+zx^{2})+c(xyx+yxy+z^{3})
			\end{align*}
			is a desired regular twisted superpotential.
			\item {\rm Type E}.
			Let
			$$w=(xzx+\eta zx^{2}+\eta^{8}x^{2}z)
			+(yxy+\eta^{4}xy^{2}+\eta^{5}y^{2}x)
			+(zyz+\eta^{7}yz^{2}+\eta^{2}z^{2}y).$$
			By taking the left derivatives, we have that
			$$\left\{
			\begin{array}{ll}
			\partial_{x}(w)=zx+\eta^{8}xz+\eta^{4}y^{2}, \\
			\partial_{y}(w)=xy+\eta^{5}yx+\eta^{7}z^{2}, \\
			\partial_{z}(w)=\eta x^{2}+yz+\eta^{2}zy.
			\end{array}
			\right.
			$$ 
			By taking the right derivatives, we have that
			$$\left\{
			\begin{array}{ll}
			(w)\partial_{x}=xz+\eta zx+\eta^{5}y^{2}=\eta^{8}\partial_{x}(w), \\
			(w)\partial_{y}=yx+\eta^{4}xy+\eta^{2}z^{2}=\eta^{5}\partial_{y}(w), \\
			(w)\partial_{z}=\eta^{8}x^{2}+zy+\eta^{7}yz=\eta^{2}\partial_{z}(w).
			\end{array}
			\right.
			$$
			By Lemma \ref{TSP}, 
			the potential $w$ is a twisted superpotential.
			By \cite[Theorem 4.9]{IM1}, we have that $\mathcal{D}(w)=\mathcal{A}(E,\sigma_{q}\tau_{2}^{2})$
			where $q=(\eta^{8}:\eta^{4}:1)$.
			
			Let
			$$w=(xzx+\eta^{8} zx^{2}+\eta x^{2}z)
			+(yxy+\eta^{5}xy^{2}+\eta^{4}y^{2}x)
			+(zyz+\eta^{2}yz^{2}+\eta^{7}z^{2}y).$$
			By taking the left derivatives, we have that
			$$\left\{
			\begin{array}{ll}
			\partial_{x}(w)=zx+\eta xz+\eta^{5}y^{2}, \\
			\partial_{y}(w)=xy+\eta^{4}yx+\eta^{2}z^{2}, \\
			\partial_{z}(w)=\eta^{8} x^{2}+yz+\eta^{7}zy.
			\end{array}
			\right.
			$$
			By taking the right derivatives, we have that
			$$\left\{
			\begin{array}{ll}
			(w)\partial_{x}=xz+\eta^{8} zx+\eta^{4}y^{2}=\eta^{8}\partial_{x}(w), \\
			(w)\partial_{y}=yx+\eta^{5}xy+\eta^{7}z^{2}=\eta^{5}\partial_{y}(w,) \\
			(w)\partial_{z}=\eta x^{2}+zy+\eta^{2}yz=\eta^{2}\partial_{z}(w).
			\end{array}
			\right.
			$$
			By Lemma \ref{TSP}, 
			the potential $w$ is a twisted superpotential.
			By \cite[Theorem 4.9]{IM1}, we have that $\mathcal{D}(w)=\mathcal{A}(E,\sigma_{q'}\tau_{2}^{4})$
			where $q'=(\eta:\eta^{5}:1)$.
			Since $q'=\tau_{2}(q)$, it follows from Theorem \ref{iso} that
			$\mathcal{D}(w) \cong \mathcal{A}(E,\sigma_{q}\tau_{2}^{4})$.
			
			\item {\rm Type H}.
			Let $p=(1:1:\la)$ where $\la=1+\sqrt{3}$.
			Since
			\begin{align*}
			&(\tau_{3} \otimes \tau_{3} \otimes \tau_{3})(w_{p})\\
			&=3\sqrt{3}(xyz+yzx+zxy)+3\sqrt{3}(xzy+yxz+zyx)+3(3+\sqrt{3})(x^{3}+y^{3}+z^{3}) \\
			&=3\sqrt{3}w_{p},
			\end{align*}
			we have that $\tau_{3} \in \Aut\,(w_{p})$.
			By Lemma \ref{tsp},
			\begin{align*}
			(w_{p})^{\tau_{3}}&=(\tau_{3}^{2} \otimes \tau_{3} \otimes {\rm id})(w_{p}) \\
			&=(\ep xyz+\la yzx+\ep^{2}zxy)+(\la xzy+\ep yxz+\ep^{2}zyx)
			+(\ep\la x^{2}y+xyx+\ep^{2}\la yx^{2}) \\
			&+(\ep^{2}x^{2}z+xzx+\ep zx^{2})+(\ep \la y^{2}x+yxy+\ep^{2}\la yx^{2})
			+(\ep^{2}y^{2}z+yzy+\ep zy^{2})\\
			&+(z^{2}x+\la zxz+xz^{2})+(z^{2}y+\la zyz+yz^{2})
			+(x^{3}+y^{3}+\la z^{3})
			\end{align*} 
			is a desired regular twisted superpotential.
			
			Since $\tau_{3} \in \Aut\,(w_{p})$, $\tau_{3}^{3} \in \Aut\,(w_{p})$.
			By Lemma \ref{tsp},
			\begin{align*}
			(w_{p})^{\tau_{3}^{3}}&=(\tau_{3}^{2} \otimes \tau_{3}^{3} \otimes {\rm id})(w_{p}) \\
			&=(\ep^{2} xyz+\la yzx+\ep zxy)+(\la xzy+\ep^{2} yxz+\ep zyx)
			+(\ep^{2}\la x^{2}y+xyx+\ep\la yx^{2}) \\
			&+(\ep x^{2}z+xzx+\ep^{2} zx^{2})+(\ep^{2} \la y^{2}x+yxy+\ep\la yx^{2})
			+(\ep y^{2}z+yzy+\ep^{2} zy^{2})\\
			&+(z^{2}x+\la zxz+xz^{2})+(z^{2}y+\la zyz+yz^{2})
			+(x^{3}+y^{3}+\la z^{3})
			\end{align*} 
			is a desired regular twisted superpotential.
		\end{enumerate}
\end{proof}

It is easy to check that
a potential $w$ in Table $2$ is a superpotential
if and only if it is of Type A, so
the following corollary follows 
from Lemma \ref{CY-superpotential} and Theorem \ref{main04}.
\begin{cor}\label{cor}
    A $3$-dimensional quadratic AS-regular
    algebra of Type EC $\mathcal{A}(E,\sigma_{p}\tau^{i})$
    is Calabi-Yau if and only if $i=0$, that is,
    it is a $3$-dimensional Sklyanin algebra.
\end{cor}
For a Sklyanin potential $w_{p}$, we define
$$
{\rm PAut}\,(w_{p}):=\Aut\,(w_{p})/(k^{\times}{\rm id}_{V})
$$
where $k^{\times}{\rm id}_{V}:=\{\alpha {\rm id}_{V} \mid \alpha \in k^{\times}\}$.
\begin{lem}\label{auto}
	Let $A=T(V)/(R)=\mathcal{A}(E,\sigma)$ be a geometric algebra.
	Suppose that there is a superpotential $w \in V^{\otimes 3}$ such that $A=\mathcal{D}(w)$.
	If $\theta \in \Aut\,(w)$, then $\overline{\theta^{\ast}} \in \Aut_{k}(\mathbb{P}(V^{\ast}),E)$.
	\end{lem}
	\begin{proof}
		Let $\theta \in \Aut\,(w)$.
		By \cite[Lemma 3.1]{MS1}, 
		$\theta$ induces a graded algerba automorphism $\phi$ of $A$.
		By \cite[Lemma 6.4]{Mo2}, 
		$\phi$ induces an automorphism 
		$\overline{\theta^{\ast}}=\overline{(\phi|_{V})^{\ast}} \in \Aut_{k}(\mathbb{P}(V^{\ast}),E)$.
	\end{proof}
	It is easy to check the following lemma.
	\begin{lem}\label{torsion}
		Let $E$ be an elliptic curve in $\mathbb{P}^{2}$,
		$p_{1}=(1:-\ep:0)$ and $p_{2}=(1:0:-1) \in E[3]$
		where $\ep$ is a primitive $3$rd root of unity.
		Then $E[3]=\{mp_{1}+np_{2} \mid 0 \leq m,n \leq 2 \}$.
	\end{lem}
	We denote by $T[3]$ the set of translations by $p \in E[3]$.
	\begin{lem}\label{trans}
		For any $p \in E \setminus E[3]$,
		$T[3] \leq {\rm PAut}\,(w_{p})$.
	\end{lem}
	\begin{proof}
		Let $p=(a:b:c) \in E \setminus E[3]$.
		Since $\sigma_{p_{1}}=\begin{pmatrix}
		\ep^{2} & 0 & 0 \\
		0 & 1 & 0 \\
		0 & 0 & \ep
		\end{pmatrix}$ and $\sigma_{p_{2}}=\begin{pmatrix}
		0 & 1 & 0 \\
		0 & 0 & 1 \\
		1 & 0 & 0
		\end{pmatrix}$ where $p_{1}=(1:-\ep:0)$ and $p_{2}=(1:0:-1)$,
		$$(\sigma_{p_{1}} \otimes \sigma_{p_{1}} \otimes \sigma_{p_{1}})(w_{p})=w_{p}$$
		and
		\begin{align*}
			(\sigma_{p_{2}} \otimes \sigma_{p_{2}} \otimes \sigma_{p_{2}})(w_{p})
			&=a(yzx+zxy+xyz)+b(yxz+zyx+yxz)+c(y^{3}+z^{3}+x^{3}) \\
			&=w_{p},
		\end{align*}
		so $\sigma_{p_{1}}, \sigma_{p_{2}} \in {\rm PAut}\,(w_{p})$.
		By Lemma \ref{torsion}, 
		$T[3] \leq {\rm PAut}\,(w_{p})$.
	\end{proof}
	\begin{prop}\label{Paut}
		Let $p \in E \setminus E[3]$.
		\begin{enumerate}[{\rm (1)}]
			\item If $j(E) \neq 0,12^{3}$, then
			$${\rm PAut}\,(w_{p})=\begin{cases}
			T[3] \,\,\,&\text{if } p \notin E[2], \\
			T[3] \rtimes \langle \tau_{1} \rangle \,\,\,&\text{if } p \in E[2].
			\end{cases}$$
		
			\item If $\la=0$, then
			$${\rm PAut}\,(w_{p})=\begin{cases}
			T[3] \,\,\,&\text{if } p \notin E[2], \\
			T[3] \rtimes \langle \tau_{2}^{3} \rangle \,\,\,&\text{if } p \in E[2].
			\end{cases}$$
			
			\item If $\la=1+\sqrt{3}$, then
			$${\rm PAut}\,(w_{p})=\begin{cases}
			T[3] \,\,\,&\text{if } p \notin E[2], \\
			T[3] \rtimes \langle \tau_{3}^{2} \rangle \,\,\,&\text{if } p \in E[2] \setminus
			\{o_{E},(1:1:\la)\}, \\
			T[3] \rtimes \langle \tau_{3} \rangle \,\,\,&\text{if } p=(1:1:\la).
			\end{cases}$$
		\end{enumerate}
	\end{prop}
	\begin{proof}
		Let $p=(a:b:c) \in E \setminus E[3]$.
		\begin{enumerate}[{\rm (1)}]
			\item Since 
				$$(\tau_{1} \otimes \tau_{1} \otimes \tau_{1})(w_{p})
				=a(yxz+xzy+zyx)+b(yzx+xyz+zxy)+c(y^{3}+x^{3}+z^{3}),$$
				$\tau_{1} \in {\rm PAut}\,(w_{p})$ if and only if $a=b$ if and only if $p \in E[2]$
				by Lemma \ref{lem1}.
				
				\item Assume that $\la=0$.
				Since $\tau_{2}^{3}=\tau_{1}$, $\tau_{2}^{3} \in {\rm PAut}\,(w_{p})$ if and only if $p \in E[2]$.
				Since 
				\begin{align*}
					&(\tau_{2} \otimes \tau_{2} \otimes \tau_{2})(w_{p}) \\
					&=a\eta^{3}(yxz+xzy+zyx)+b\eta^{3}(yzx+xyz+zxy)+c(y^{3}+x^{3}+z^{3}), \\
					&(\tau_{2}^{2} \otimes \tau_{2}^{2} \otimes \tau_{2}^{2})(w_{p}) \\
					&=a\eta^{6}(xyz+yzx+zxy)+b\eta^{6}(xzy+yxz+zyx)+c(x^{3}+y^{3}+z^{3}),
				\end{align*}
				we have that $\tau_{2}, \tau_{2}^{2} \notin {\rm PAut}\,(w_{p})$.
				
				\item Assume that $\la=1+\sqrt{3}$.
				Since $\tau_{3}^{2}=\tau_{1}$, $\tau_{3}^{2} \in {\rm PAut}\,(w_{p})$ if and only if
				$p \in E[2]$.
				Since
				\begin{align*}
					&(\tau_{3} \otimes \tau_{3} \otimes \tau_{3})(w_{p}) \\
					&=3(a\ep+b\ep^{2}+c)(xyz+yzx+zxy)+3(a\ep^{2}+b\ep+c)(xzy+yxz+zyx) \\
					&+3(a+b+c)(x^{3}+y^{3}+z^{3}),
				\end{align*}
				$\tau_{3} \in {\rm PAut}\,(w_{p})$ if and only if
				$$\tau_{3}^{3}(p)=(a\ep+b\ep^{2}+c:a\ep^{2}+b\ep+c:a+b+c)=(a:b:c)=p$$
				if and only if $p \in E_{\tau_{3}^{3}}$.
				By Proposition \ref{main01}, $E_{\tau_{3}^{3}}=\langle (1:1:\la) \rangle$, so
				if $p \neq o_{E}$, then
				$p \in E_{\tau_{3}^{3}}$ if and only if $p=(1:1:\la)$.
				Thus we have that $\tau_{3} \in \Aut\,(w_{p})$ if and only if $p=(1:1:\la)$.
		\end{enumerate}
	\end{proof}
Let $p \in E \setminus E[3]$.
Recall that a Sklyanin algebra $\mathcal{A}(E,\sigma_{p})$
is a $3$-dimensional quadratic Calabi-Yau AS-regular algebra
by Corollary \ref{cor}.
Finally, we prove Theorem \ref{thm} in Introduction.
\begin{thm}\label{main05}
    Let $A$ be a $3$-dimensional quadratic AS-regular algebra.
    Then there are a $3$-dimensional quadratic Calabi-Yau AS-regular algebra $S$
    and $\phi \in {\rm GrAut}_{k}\,S$ such that $A \cong S^{\phi}$
    if and only if $A$ is not of Type E.
\end{thm}
\begin{proof}
	If $A$ is not a Type EC algebra, then the result follows from \cite[Theorem 4.4]{IM2},
	so we assume that $A$ is a Type EC algebra.
	
	If $A$ is not of Type E, then $A \cong \mathcal{D}(w_{p})^{\tau^{i}}$
	where $\mathcal{D}(w_{p})$ is a $3$-dimensional quadratic 
	Calabi-Yau AS-regular algebra and $\tau^{i} \in \Aut\,(w_{p}) \subset {\rm GrAut}_{k}\,\mathcal{D}(w)$
	by Lemma \ref{twist}, the proof of Theorem \ref{main04} and \cite[Lemma 3.1]{MS1}.
	
	Conversely, let $A$ be of Type E.
	By Theorem \ref{main03} we may assume that $\la=0$ and
	$A=\mathcal{A}(E,\sigma_{q}\tau_{2}^{2i})$
	where $q=(\eta^{8}:\eta^{4}:1)$, $\eta$ is a primitive $9$th root of unity
	and $i=1,2$.
	Assume that there are a $3$-dimensional quadratic Calabi-Yau AS-regular algebra $S$
	and $\phi \in {\rm GrAut}_{k}\,S$ such that $A \cong S^{\phi}$.
	Since $S=\mathcal{D}(w_{p})$ by Corollary \ref{cor},
	we have that $\phi|_{V} \in \Aut\,(w_{p})$ by Lemma \ref{twist}.
By Proposition \ref{Paut},
there are $r \in E[3]$ and $j=0,1$ such that $\overline{(\phi|_{V})^{\ast}}=\sigma_{r}\tau_{2}^{3j}$.
It follows from \cite[Proposition 2.6]{IM1} that
$\mathcal{D}(w_{p})^{\phi}=\mathcal{A}(E,\sigma_{p}\overline{(\phi|_{V})^{\ast}})
=\mathcal{A}(E,\sigma_{p+r}\tau_{2}^{3j})$.
By Theorem \ref{iso}, $A \not\cong \mathcal{A}(E,\sigma_{p+r}\tau_{2}^{3j})$,
a contradiction.

\end{proof}
\begin{rem}
	Since every Type EC algebra is graded Morita equivalent to
	a Calabi-Yau AS-regular algebra by \cite[Theorem 4.4]{IM2},
	a Type E algebra is a twist of a Calabi-Yau AS-regular algebra by
	a twisting system introduced by \cite{Z},
	but not by a graded algebra automorphism by Theorem \ref{main05}.
\end{rem}
\section*{Acknowledgments}
I am grateful to my supervisor Izuru Mori
for his advice and helpful discussions.

\end{document}